\documentclass[a4paper,leqno,11pt,twoside]{amsart}
\usepackage{amsmath,amsfonts,amssymb,amsthm,amscd}
\usepackage[utf8]{inputenc}
\usepackage[english]{babel}
\usepackage[colorlinks=true,citecolor=blue,urlcolor=blue,linkcolor=blue]{hyperref}
\usepackage[top=1.1in,bottom=1.1in,left=1in,right=1in]{geometry}
\usepackage{graphicx}
\usepackage{tikz-cd}
\usepackage{mathtools}
\usepackage{xcolor}
\usepackage{aliascnt}

\theoremstyle{plain}

\newtheorem{thmIntr}{Theorem}

\newaliascnt{propIntr}{thmIntr}
\newtheorem{propIntr}[propIntr]{Proposition}

\newaliascnt{corIntr}{thmIntr}
\newtheorem{corIntr}[corIntr]{Corollary}

\newaliascnt{lem}{thm}
\newtheorem{lem}[lem]{Lemma}
\aliascntresetthe{lem}

\newaliascnt{cor}{thm}

\aliascntresetthe{cor}

\newaliascnt{prop}{thm}
\newtheorem{prop}[prop]{Proposition}
\aliascntresetthe{prop}

\theoremstyle{definition}

\newaliascnt{rem}{thm}
\newtheorem{rem}[rem]{Remark}
\aliascntresetthe{rem}

\newaliascnt{defn}{thm}

\aliascntresetthe{defn}

\newaliascnt{setup}{thm}
\newtheorem{setup}[setup]{Set-up}
\aliascntresetthe{setup}

\numberwithin{equation}{section}

\def\bC{\ensuremath{\mathbb{C}}}
\def\bP{\ensuremath{\mathbb{P}}}
\def\bQ{\ensuremath{\mathbb{Q}}}
\def\bZ{\ensuremath{\mathbb{Z}}}

\def\cG{\ensuremath{\mathcal{G}}}
\def\cI{\ensuremath{\mathcal{I}}}
\def\cM{\ensuremath{\mathcal{M}}}
\def\cO{\ensuremath{\mathcal{O}}}
\def\cU{\ensuremath{\mathcal{U}}}

\def\wC{\ensuremath{\widetilde{C}}}

\DeclareMathOperator{\Pic}{Pic}
\DeclareMathOperator{\Hom}{Hom}

\DeclareMathOperator{\irr}{irr}
\DeclareMathOperator{\CH}{CH}

\title[Hyperelliptic curves on $K3$ surfaces]{Hyperelliptic curves on $K3$ surfaces}
\author{Federico Moretti}
\hypersetup{
  pdftitle={Hyperelliptic curves on K3 surfaces},
  pdfauthor={Federico Moretti},
  pdfkeywords={K3 surfaces, hyperelliptic curves, Hilbert squares, minimal model program, kernel bundles}
}

\address{Stony Brook University, 100 Nicolls Road, Stony Brook, NY 11794}
\email{federico.moretti@stonybrook.edu}

\begin{document}

\begin{abstract}
	We study hyperelliptic curves on $K3$ surfaces. We show that any number of general points on a $K3$ surface of Picard rank one can be connected by a hyperelliptic curve. As a by-product we recover the Beauville--Voisin decomposition of the small diagonal, together with some new bounds on its torsion order in the integral Chow group. On the arithmetic side, given a $K3$ surface $X$ of geometric Picard rank one defined over a number field $k$, we show that there exists a finite extension $k\subset k'$ such that quadratic points are dense in $X_{k'}$, extending a theorem of Corvaja--Zucconi to $K3$ surfaces of arbitrary degree.
\end{abstract}


\maketitle

\section*{Introduction}
\quad The purpose of this note is to show that there are surprisingly many hyperelliptic curves on a $K3$ surface.

\quad Let $X$ be a complex projective $K3$ surface of Picard rank $1$. A classical result by Bogomolov and Mumford \cite{morimukai} states that any complete linear series on $X$ contains a chain of (singular) rational curves. On the primitive linear series of a very general $K3$ surface this rational curve can be chosen nodal by a result of Chen \cite{chenrational}. It follows that any such linear series contains a two-dimensional family of curves of geometric genus $2$, and hence hyperelliptic. Our first result shows that $X$ actually contains many other hyperelliptic curves:

\begin{thmIntr}\label{irrCor}
	Let $X$ be a complex $K3$ surface of Picard rank one. Then any number of general points of $X$ can be connected by a hyperelliptic curve\footnote{Here and throughout, a hyperelliptic curve on $X$ is an irreducible curve whose normalization is hyperelliptic.}.
\end{thmIntr}

\quad This is equivalent to the fact that $X$ is dominated by a double cover of $\mathbb P^2$, see \cite[Thm.~C]{CLM}.

\quad Over a number field, the families of hyperelliptic curves we construct carry dense quadratic points after a finite field extension. In particular, we deduce:

\begin{thmIntr}\label{arithmeticIntro}
	Let $X$ be a $K3$ surface of geometric Picard rank one defined over a number field. After a finite field extension, quadratic points are Zariski dense in $X$.
\end{thmIntr}

 \quad Recall that a variety over a number field is said to be potentially dense if its rational points become Zariski dense after a finite field extension. It is a folklore conjecture that every $K3$ surface is potentially dense, and \autoref{arithmeticIntro} can be seen as a weak form of this conjecture, in which rational points are replaced by quadratic points\footnote{We do not claim any real progress towards the potential density conjecture.}.

\quad The two theorems above fit naturally with several invariants of geometric and arithmetic nature that have recently attracted some attention. Recall that the degree of irrationality $\irr(Y)$ of a variety $Y$ is the least degree of a dominant rational map from $Y$ to projective space of the same dimension; it measures how far $Y$ is from being rational, and we refer to \cite{primer} for an excellent survey. For varieties of dimension at least two it is well understood only in a few cases, such as hypersurfaces \cite{bdelu} and abelian surfaces \cite{chenabelian,martinabelian,jiangabelian}. The degree of irrationality of a $K3$ surface of Picard rank one is conjectured to go to infinity with the degree of the polarization. By contrast, \autoref{irrCor} shows that $X$ is dominated by a surface $Y$ with $\irr(Y)=2$. Bastianelli and Yang studied the degree of irrationality of covers of surfaces in $\mathbb P^3$ and of very general hypersurfaces in $\mathbb P^{n+1}$ \cite{bastianelli,yang}, and \autoref{irrCor} can be seen as an analogue of their results for $K3$ surfaces. In the language of \cite{CLM}, \autoref{irrCor} says that the $p$-connecting gonality of $X$ is equal to $2$ for every $p\geq2$. In fact, this question grew out of discussions with Nathan Chen and Ben Church, where we explored the possibility of obstructing low degree maps from $K3$ surfaces to $\mathbb P^2$ by counting how many general points curves of a given gonality can connect.

\quad On the arithmetic side, \autoref{arithmeticIntro} extends \cite[Theorem~A]{corvajazucconi}, where the analogous statement is proved for quartic surfaces, to $K3$ surfaces of arbitrary degree. Following Viray and Vogt \cite[Defs.~2.2 and 2.3]{virayvogt}, it says that $2\in\wp(X/k)$, where
\[
\wp(X/k)=\{e\in\mathbb N \mid \text{points of degree $e$ are dense in $X$ after a finite field extension}\}
\]
is the potential density degree set. The minimum of this set has recently been proposed as an arithmetic counterpart of $\irr$ \cite[\S4.6]{primer}, and from this point of view \autoref{arithmeticIntro} gives $\min\wp(X/k)\leq 2$; the potential density conjecture predicts that this minimum is $1$. More generally, this fits into the conjectural picture of Harris and Tschinkel \cite[Conj.~1.4]{harristschinkel}, predicting potential density for smooth projective varieties with nef anticanonical bundle. The conjecture is known for elliptic $K3$ surfaces by \cite{bogomolovtschinkel}. When $d$ is a square, \autoref{arithmeticIntro} also follows from the main theorem of \cite{hassetttschinkelsym}, where Hassett and Tschinkel prove, among other things, potential density for $X^{[2]}$ if $d=k^2$. A remarkable related case is due to Amerik and Voisin \cite{amerikvoisin}, who establish potential density on the variety of lines of a sufficiently general cubic fourfold.

\quad Both theorems come from the geometry of the Hilbert square $X^{[2]}$.  We study rational surfaces $P\subset X^{[2]}$ whose incidence correspondence dominates $X$. The universal family over $P$ is then a double cover of a rational surface dominating $X$ and the hyperelliptic curves of \autoref{irrCor} are obtained by pulling back rational curves in $P$. Our main result is that such surfaces always exist.

\begin{thmIntr}\label{mainthm}
	Let $(X,L)$ be a complex polarized $K3$ surface of Picard rank $1$. Then there is a rational surface $P\subset X^{[2]}$ whose incidence correspondence dominates $X$.
	\label{A}
\end{thmIntr}
\quad It should come as no surprise that such a rational surface has interesting consequences at the level of rational and integral Chow groups. In particular, the induced correspondence can be used to recover the decomposition of the small diagonal $\Delta_{123}\subset X^3$ due to Beauville and Voisin \cite{beauvillevoisin}, and it gives new bounds on the torsion order of this decomposition in the integral Chow group. Recall that every $K3$ surface carries a canonical class $c_X\in \CH_0(X)$, defined as the class of a point on a rational curve on $X$. We set $c_i:=\operatorname{pr}_i^*c_X\in\CH^2(X^3)$, and we denote by $\Delta_{ij}\subset X^3$ the big diagonals.

\begin{thmIntr}[Beauville--Voisin]\label{diagonalIntro}
	Let $X$ be a complex projective $K3$ surface. Then the following decomposition holds in $\CH^4(X^3)_{\mathbb Q}$:
	\begin{equation*}
		[\Delta_{123}]
		=\big([\Delta_{12}]-c_1\big)\cdot c_3
		+\big([\Delta_{23}]-c_2\big)\cdot c_1
		+\big([\Delta_{13}]-c_3\big)\cdot c_2 .
	\end{equation*}
\end{thmIntr}

\quad  The starting observation here was suggested to us by Olivier Martin. Our proof goes along the lines of \cite{bazhov}; it works with integral coefficients up to the very last step and shows that the relation holds in $\CH^4(X^3)$ up to $2N$-torsion, where $N$ is the degree of the incidence correspondence over $X$. In several cases one can find two surfaces $P_1$ and $P_2$ whose incidence correspondences have coprime degrees, and then the relation holds up to $2$-torsion; this happens, for instance, in degree $2$ and for infinitely many other values of the degree, see \autoref{coprimm} and \autoref{remark}. Since the degree of irrationality of $X$ is expected to grow with the degree, while the decomposition holds up to $2$-torsion in arbitrarily high degree, the torsion order of the integral decomposition of the small diagonal most likely cannot be used to obstruct low degree rational maps from $X$ to $\mathbb P^2$. 

\quad As a by-product, the intersection of $P$ with the exceptional divisor produces new examples of constant-cycle curves, i.e.\ curves whose points all have class $c_X$, introduced and studied by Huybrechts \cite{huybrechtsccc}. In the following $E_{HC}\subset X^{[2]}$ denotes the Hilbert--Chow exceptional divisor (i.e. the locus of non-reduced subschemes).
\begin{corIntr}
	Every component of the image of $P\cap E_{\mathrm{HC}}$ is a constant-cycle curve.
\end{corIntr}

\quad All the proofs are fairly elementary, modulo some deep results on the birational geometry of $X^{[2]}$. The proof of \autoref{mainthm} relies on foundational work of Bayer and Macrì \cite{bayermacri}, and \autoref{irrCor} and \autoref{arithmeticIntro} follow almost immediately. The proof of \autoref{diagonalIntro} is more involved and follows an argument of Bazhov \cite{bazhov}, who gave a different proof of the decomposition using another auxiliary correspondence; here it is replaced by the incidence correspondence over $P$.

\quad Here is a sketch of how to produce the rational surface of \autoref{A}. The two boundary rays of the movable cone of $X^{[2]}$ produce the Hilbert--Chow contraction and, up to a small birational modification, either another divisorial contraction, whose exceptional divisor is a fibration in rational curves over a $K3$ surface\footnote{A (possibly twisted) Fourier--Mukai partner of $X$.} $S$, or a Lagrangian fibration \cite[Thm.~5.7 and \S13]{bayermacri}. In the first case we restrict this conic bundle to a rational curve in $S$; in the second we take a rational fibre of the Lagrangian fibration. For a sparse value of the degrees a Mukai flop is needed to reach the model carrying the contraction. In this case there is a second option for the rational surface $P$: the centre of the flop is a Lagrangian plane, hence already a rational surface in $X^{[2]}$. A cohomological computation shows that in all these cases the incidence correspondence dominates $X$.

\quad Applied blindly, the Bayer--Macrì machinery produces the rational surfaces of \autoref{mainthm} almost magically, but it says little about the hyperelliptic curves themselves. In concrete cases, however, these surfaces and the curves they parametrize can often be described quite explicitly; we describe some of them in Section~\ref{sec:bundles}.  This was in fact our starting point, before diving into the beautiful literature on the MMP of $X^{[2]}$. In the examples above (and generalizing a bit the concept of stable bundle
 also in \cite{bayermacri}) the hyperelliptic pencils appear as pencils of morphisms between two vector bundles on $X$ (or, more generally, two stable objects), each morphism degenerating exactly along a length-two subscheme.  Using a twist of arguments already present in the literature, namely a kernel bundle construction in the spirit of \cite{lazarsfeld}, which to some extent is hidden in many of the papers mentioned above, we realize the hyperelliptic pencil $|A|$ on a curve $C\subset X$ as a line in $\mathbb P\mathrm{Hom}(E,F)$ for two vector bundles $E,F$ of consecutive rank on $X$ with top Chern class satisfying $c_2(F)=c_2(E)+2$ (each of these morphisms drops rank at $2$ points). In this way we recover the families of \cite[Props.~7.2 and 7.7]{flaminiknutsenpacienza}, starting from the nodal hyperelliptic curves constructed by Ciliberto and Knutsen \cite{cilibertoknutsen}.

\begin{propIntr}\label{primitivethm}
	Let $(X,L)$ be a very general primitively polarized $K3$ surface of genus $g$, and let $r\geq1$. If $g=r(r+1)$, then $X$ carries a rational covering family of nodal hyperelliptic curves in $|L|$ of geometric genus $2r$. If $g=(r+1)^2$, it carries such a family of geometric genus $2r+1$.
\end{propIntr}

\quad By \cite{cilibertoknutsen}, these are the largest geometric genera of hyperelliptic curves in $|L|$. The point of \autoref{primitivethm} is that the rational families are obtained directly from the bundles $E$ and $F$, without appealing to the Bayer--Macrì description.

\smallskip

\quad Besides the work of Bayer and Macrì, the birational geometry of $X^{[2]}$, its rational curves and the connection with hyperelliptic curves in $X$ were studied (at least implicitly) by many authors, among them \cite{flaminiknutsenpacienza,cilibertoknutsen,hassetttschinkelrc,hassetttschinkelcones,markmansurvey,bayer,bayerhassetttschinkel,mongardi}. Lagrangian planes (that can be used directly as a surface in \autoref{A} for certain degrees) were studied for example in \cite[\S5]{hassetttschinkelcones}, \cite{oberdieck} and \cite{bakkerplanes}.

\quad The paper is organized as follows. Sections~\ref{sec:MMP}--\ref{sec:mainproof} contain the MMP observation and the incidence calculation. Section~\ref{sec:bundles} gives the vector-bundle picture and is independent of the previous sections. Finally, Section~\ref{sec:consequences} proves the applications stated above.

{\bf Acknowledgements and AI use.} The perspective developed in this paper has been especially important to me because it led to many rich mathematical interactions. I am very grateful to Rob Lazarsfeld for numerous discussions around this topic, for encouraging me to write this note, and for giving several suggestions on how to restructure the introduction. I was led to the vector-bundle viewpoint used here while working with Andr\'es Rojas on \cite{morettirojas}; I am deeply grateful to Andr\'es, who taught me much of what I know about derived categories of $K3$ surfaces and with whom I also discussed possible approaches to this problem at Oberwolfach in the summer of 2025. I thank Andreas Leopold Knutsen for several discussions about \cite{cilibertoknutsen} and for sharing many useful ideas about Lazarsfeld--Mukai bundles for nodal curves. I am also grateful to Francesco Bastianelli, Nathan Chen, Ben Church, Pietro Corvaja, Phil Engel, Olivier Martin, Jieao Song, and Claire Voisin who shared valuable insights and comments that shaped much of the introduction. 

AI tools were used for editorial assistance: organizing the exposition, improving the language,  literature search, and the numerical search in \autoref{remark}. All the mathematics was developed by the author, mainly growing from discussions with the aforementioned people. The author takes full responsibility for every statement, citation, and argument.

\section{The MMP for the Hilbert square}\label{sec:MMP}

In this section we give some preliminaries on the MMP for the Hilbert square of a $K3$ surface, following Bayer--Macr\`i \cite[\S 3,\S 5,\S 12,\S 13]{bayermacri}. For generalities on hyper-K\"ahler manifolds, we refer to \cite{debarreHK}.
\subsection{The lattice}

Write
\[
L^2=2d,\qquad M=X^{[2]}.
\]
Let $h=\widetilde L$ be the divisor induced by $L$, and let $E_{\mathrm{HC}}$ be the exceptional
divisor of the Hilbert--Chow morphism
\[
M\longrightarrow X^{(2)},
\]
so that $E_{\mathrm{HC}}=2\delta$ for a primitive class $\delta\in\operatorname{NS}(M)$. By
Beauville \cite[\S6]{beauville} there is an orthogonal direct sum decomposition
\[
H^2(M,\bZ)=H^2(X,\bZ)\oplus\bZ\delta
\]
with respect to the Beauville--Bogomolov form $q$ (see for instance \cite[\S 3.1]{debarreHK}), the form restricts to the intersection form on
$H^2(X,\bZ)$, and $q(\delta)=-2$. Since $\Pic(X)=\bZ L$, this gives
\[
\operatorname{NS}(M)=\bZ h\oplus\bZ\delta
\]
and
\begin{equation}\label{BBform}
q(h)=2d,\qquad q(\delta)=-2,\qquad (h,\delta)=0.
\end{equation}

Recall that the \emph{movable cone} $\mathrm{Mov}(M)\subset \mathrm{NS}(M)_{\mathbb R}$ is the closure of the cone generated by classes of divisors whose stable base locus has codimension at least two. 
One boundary of the movable cone of $M$ is spanned by $h$ and carries the Hilbert--Chow
contraction. We recall the class of the second boundary and the associated contraction in the two
cases needed below.

\subsection{The non-square case: a divisorial contraction}\label{subsec:nonsquareMMP}

Assume that $d$ is not a square. Bayer and Macr\`i determine the movable cone of
$\operatorname{Hilb}^n(X)$ for $\Pic(X)=\bZ \cdot L$ in terms of two Pell equations
\cite[\S13]{bayermacri}: their (32),
$(n-1)x^2-dy^2=1$, which governs Brill--Noether contractions, and their (33), $x^2-d(n-1)y^2=1$ with
$n-1\mid x+1$, which governs the Hilbert--Chow and the Li--Gieseker--Uhlenbeck contractions. For
$n=2$ both reduce to
\begin{equation}\label{pellone}
x^2-dy^2=1,
\end{equation}
which has a solution whenever $d$ is not a square. By \cite[Prop.~13.1(b)]{bayermacri} the second
boundary ray of the movable cone is therefore generated by
\begin{equation}\label{Aclass}
A=xh-dy\delta,
\end{equation}
where $(x,y)$ is the positive solution of \eqref{pellone} with $x$ smallest. By \eqref{BBform} it
has positive square $q(A)=2d(x^2-dy^2)=2d$. The
primitive class orthogonal to it is
\begin{equation}\label{eclass}
e=yh-x\delta,
\end{equation}
and \eqref{BBform} gives
\[
q(e)=2(dy^2-x^2)=-2,\qquad (A,e)=0.
\]

By \cite[Thms.~12.3 and 12.4]{bayermacri} the movable cone of $M$ is cut out  by the hyperplanes orthogonal to the irreducible exceptional divisors. Since
$q(A)=2d>0$ and $\rho(M)=2$, the line $\mathbb{R}A$ is one of these hyperplanes, so $e$ spans the ray of
an irreducible exceptional divisor $E$. The square of a prime exceptional divisor on a manifold
of $K3^{[2]}$ type is $-2$ or $-8$ (see for example \cite[Theorem 3.1]{vangeemenkapustka}); since $q(\epsilon e)=-2\epsilon^2$, this means that, for some $\epsilon\in\{1,2\}$,
\begin{equation}\label{exceptionalclass}
[E]=\epsilon e=\epsilon(yh-x\delta).
\end{equation}
By \cite[Thm.~5.7(a)]{bayermacri}  the ray of $A$ induces a divisorial
contraction: on the birational model $M'$ on which $A$ is nef, a positive multiple of $A$ defines
$
	\pi:M'\longrightarrow \overline M,$
and the divisor it contracts is the transform of $E$. By \cite[Thm.~5.7(b)]{bayermacri} the
birational transformation $M\dashrightarrow M'$ is a series of flopping contractions (centered along Lagrangian planes), therefore it is an isomorphism
in codimension one. The
image $S:=\pi(E)$ is a $K3$ surface, a (possibly twisted) Fourier--Mukai partner of $X$ and hence of Picard rank one,
and $\pi$ induces a $\bP^1$-fibration $\varphi:E\to S$; see for example
\cite[\S3.1]{vangeemenkapustka}.\footnote{This fibration is a smooth conic bundle and need not be
the projectivization of a rank two bundle, the obstruction being a two-torsion Brauer class on $S$.}

\subsection{The square case: a Lagrangian fibration}\label{subsec:squareMMP}

Assume that
\[
d=k^2.
\]
By \cite[Prop.~13.1(a)]{bayermacri} the second movable boundary is generated by the primitive
isotropic divisor
\begin{equation}\label{Dclass}
D=h-k\delta,
\qquad q(D)=2d-2k^2=0.
\end{equation}
Bayer and Macr\`i prove the hyper-K\"ahler SYZ conjecture for moduli spaces of sheaves on $K3$
surfaces \cite[Thm.~1.5]{bayermacri}: a primitive isotropic nef class on a birational model induces
a Lagrangian fibration. Applied to $D$ this gives a birational model $M'$ of $M$, again isomorphic
to it in codimension one, with a morphism
\begin{equation}\label{lagfibration}
f:M'\longrightarrow\bP^2,
\qquad D=f^*\cO_{\bP^2}(1).
\end{equation}

Sawon's description makes this fibration concrete \cite{sawon}. There is a Fourier--Mukai partner
$(X',H)$ with
\[
\Pic(X')=\bZ H,\qquad H^2=2,
\]
and $M$ is a twisted relative compactified Jacobian of the genus-two curves in the primitive
linear system
\[
|H|\simeq\bP^2.
\]

When $k=1$, the polarization $L$ has degree two and $X$ is a double cover of $\bP^2$. If $k> 1$ the Lagrangian fibration is defined directly on $M$.

\subsection{The incidence calculation}\label{sec:incidence}
We record a simple cohomological computation to understand when the incidence correspondence over surface in $X^{[2]}$ dominates $X$.
Let
\[
\cU\subset X\times M
\]
be the universal family, with projections
\[
\begin{tikzcd}[column sep=large]
& \cU \arrow[dl,"q"'] \arrow[dr,"p"] & \\
X && M.
\end{tikzcd}
\]
The map $p$ is finite of degree two. The variety $\cU$ is the blow-up of $X\times X$ along the diagonal. Let
\[
\rho:\cU=\operatorname{Bl}_{\Delta}(X\times X)\longrightarrow X\times X
\]
be the blow-down, let $E_\Delta$ be its exceptional divisor, and write $l_i$ for the pullback of $L$ from the $i$-th factor. Then
\begin{equation}\label{pullbackclasses}
p^*h=l_1+l_2,
\qquad
p^*\delta=E_\Delta.
\end{equation}

\begin{lem}\label{incidencelemma}
Let $P\subset M$ be a surface with cohomology class
\[
[P]=a h^2+b h\delta+c\delta^2.
\]
Then
\begin{equation}\label{incidencedegree}
q_*p^*[P]=(2da-c)[X].
\end{equation}
In particular, the universal family over $P$ dominates $X$ if and only if $2da-c>0$.
\end{lem}

\begin{proof}
The first identity in \eqref{pullbackclasses} gives
\[
q_*p^*(h^2)=q_*(l_1+l_2)^2=2d[X].
\]
Since $\rho_*E_\Delta=0$, one has
\[
q_*p^*(h\delta)=0.
\]
Finally, the center of the blow-up has codimension two, so $E_\Delta$ is a $\bP^1$-bundle over $\Delta$ and $E_\Delta^2$ is the class of $\cO(-1)$ on it, which has degree $-1$ on each fibre. Therefore
\[
\rho_*(E_\Delta^2)=-[\Delta].
\]
Projection to the first factor gives
\[
q_*p^*(\delta^2)=-[X].
\]
Combining the three identities proves \eqref{incidencedegree}. The left-hand side is the push-forward of the universal surface over $P$, so its coefficient is the degree of the evaluation map when that map is dominant and is zero otherwise.
\end{proof}

We will use the linear functional
\begin{equation*}\label{incidencefunctional}
\ell(a h^2+b h\delta+c\delta^2)=2da-c.
\end{equation*}
 It is positive exactly when the corresponding incidence surface dominates $X$.

\section{Proof of the main theorem}\label{sec:mainproof}

We now combine the MMP with the incidence calculation. 

\subsection{The non-square case}

Assume that $d$ is not a square and retain the notation of \autoref{subsec:nonsquareMMP}. Let $S$ be the $K3$ surface at the center of the divisorial contraction $\pi:M'\to \overline M$. By the Bogomolov--Mumford theorem, $S$ contains an integral rational curve generating $\mathrm{Pic}(S)$; see \cite[Chapter 13, Thm.~1.1]{huybrechtsbook}. Choose one,
\[
\Gamma\subset S,
\]
and restrict $E$ over its normalization. The result is a surface fibred in rational curves over
$\bP^1$, hence a rational surface; its image is a rational surface
\[
P\subset E\subset M'.
\]
We first use directly the incidence computation when $M=M'$.
Let $H_S$ be the ample generator of $\Pic(S)$ and $\varphi=\pi|_E:E\to S$. We have by construction, $A=\pi^*\overline A$ with $\overline A$ ample on $\overline M$, and $\overline A|_S=w H_S$, for some integer $w$. Since $P=\varphi^{-1}(\Gamma)$, the projection formula gives,

\begin{equation*}\label{Pclassnonsquare}
[P]=\lambda A[E]
=\lambda\epsilon\left(
xyh^2-(x^2+dy^2)h\delta+dxy\delta^2
\right),
\end{equation*}
with $\lambda=\frac{1}{w}$.
 The value of the incidence functional is
\begin{equation}\label{Pdegreenonsquare}
\ell([P])
=\lambda\epsilon(2dxy-dxy)
=\lambda\epsilon dxy>0.
\end{equation}

We now explain how to conclude when $M'\neq M$. For $p\in X$
set
\[
Z_p=\{\xi:\ p\in\operatorname{Supp}\xi\}\simeq\operatorname{Bl}_pX .
\]
By \autoref{incidencelemma} the number $\ell(\alpha)$ is the number of subschemes of a cycle of
class $\alpha$ through a general point of $X$, that is $\alpha\cdot[Z_p]$. The flopping locus of
$M\dashrightarrow M'$ is a finite union of Lagrangian planes \cite{hassetttschinkelcones}. If one of
them has dominant incidence we take it as $P$ and conclude. Otherwise the incidence of each is a
curve in $X$, so for general $p$ the surface $Z_p$ misses the flopping locus and maps
isomorphically onto a surface of $M'$ along which $h$ and $\delta$ restrict as before. Hence
\eqref{Pdegreenonsquare} counts the subschemes in $P$ through a general point of $X$, so the
universal family over $P$ dominates $X$, and so does the one over its strict transform in
$X^{[2]}$, which is a rational surface. This proves \autoref{mainthm} in the non-square case.

\begin{rem}
	Using the cohomology class of the Lagrangian plane (see \cite{oberdieck,hassetttschinkelcones,bakkerplanes}) it is possible to see directly that a Lagrangian plane $\Pi$ always induces a dominant incidence correspondence of degree $N_\Pi$. The strict transform $P_1$ of one of the fibered rational surfaces $P\subset M'$ after a flop around $\Pi$ has incidence degree $N_{P_1}=\ell([P])-kN_\Pi$ for some integer $k$. Therefore if $N_\Pi$ and $\ell([P])$ are coprime both have dominant correspondence. 
	\label{coprimm}
\end{rem}
\subsection{The square case}

Assume that $d=k^2$ with $k\geq2$. Use the Lagrangian fibration \eqref{lagfibration} and its description by Sawon. The Bogomolov--Mumford theorem writes the primitive ample class $H$ on $X'$ as an effective sum of rational curves; see \cite[Cor.~13.1.5]{huybrechtsbook}. Since $\Pic(X')=\bZ H$ and $H$ is primitive, this sum consists of one integral rational curve
\[
C_0\in|H|.
\]

Let
\[
P=f^{-1}([C_0]).
\]
This fibre is a torsor under the compactified Jacobian of $C_0$, and over $\bC$ such a torsor is trivial. Thus $P$ is a rational surface. Since $D=f^*\cO_{\bP^2}(1)$, its class is
\begin{equation*}\label{Pclasssquare}
[P]=D^2=h^2-2kh\delta+k^2\delta^2.
\end{equation*}
The incidence degree is
\begin{equation*}\label{Pdegreesquare}
\ell([P])=2d-k^2=k^2=d>0,
\end{equation*}
so the universal family over $P$ dominates $X$ by \autoref{incidencelemma}.

If $k=1$, the morphism defined by $L$ is a double cover
\[
X\longrightarrow\bP^2.
\]
The fibres of the double-cover map define a plane $P\subset X^{[2]}$ whose universal family is $X$ itself. This completes the proof of \autoref{mainthm}.

\section{Some explicit examples and a vector-bundle picture}\label{sec:bundles}
In several cases one can explicitly describe a rational family of hyperelliptic curves inducing a rational surface in $X^{[2]}$.
\begin{enumerate}
	\item If $X\subset\mathbb P^3$ is a quartic, there is a natural involution, originally described by Beauville, interchanging the boundary divisors of the movable cone of $X^{[2]}$: given $\xi\in X^{[2]}$, one takes the residual intersection of the line $l_\xi$ with $X$. Here $S=X$, and the rational family of hyperelliptic curves is parametrized by a rational curve in $X$: the residual intersections of the tangent lines at a point $p\in X$ describe a curve together with a $g^1_2$, and for a general choice of $p$ this curve has genus $2$ with a node at $p$. This family was considered for example in \cite{corvajazucconi}.
	\item If $X\subset \mathbb P^3$ is a $(2,3)$ complete intersection, then $X$ carries a rank two bundle $E$ with $h^0(E)=4$ and $c_2(E)=3$ (induced by the quadric, whose section vanish along the lines on the quadric). The sections in the pencil $H^0(E\otimes \mathcal I_p)$ describe a curve nodal at $p$ together with a $g^1_2$ induced by the moving part of $Z(s)$. When $p$ moves in a rational curve in $S$ one gets a rational family of hyperelliptic curves (again these curves have a single node at $p$).
	\item If $X=Q_1\cap Q_2\cap Q_3\subset\mathbb P^5$ has degree $8$, there is an explicit Lagrangian fibration: $\xi\in X^{[2]}$ is sent to the pencil of quadrics of the net containing the line $\ell_\xi$ spanned by $\xi$, giving $f\colon X^{[2]}\to\check{\mathbb P}^2$ with $f^*\mathcal O(1)=h-2\delta$. The fibre over a pencil $\lambda$ is the Fano surface of lines of the threefold $Y_\lambda=\bigcap_{Q\in\lambda}Q$, which is the Jacobian of the genus-$2$ curve double covering $\lambda$ branched along its intersection with the discriminant sextic 
	 see  \cite[Ex.~0.9]{mukai84}. One then take as a rational surface the Jacobian of a nodal rational curve (i.e. a double cover of a pencil $\lambda$ which is a bitangent to the disciminant sextic).
	\item If $X\subset\mathrm{Gr}(2,5)$ is a general $K3$ surface of degree $10$, let $E$ be the restriction of the tautological rank $2$ bundle, so that $c_2(E)=4$. For a general $\xi\in X^{[2]}$, the residual part of the zero locus of the unique section of $E\otimes\mathcal I_\xi$ defines a birational involution of $X^{[2]}$, again exchanging the two exceptional divisors. This involution was first observed by O'Grady \cite{ogrady}.
	\item When $X^{[2]}$ is birational to the variety of lines of a cubic fourfold, such a rational surface is produced by the construction of \cite[Prop.~4.4]{voisincoisotropic}: one can take the surface of lines contained in a hyperplane section with $5$ nodes. This happens for infinitely many, though sparse, values of the degree.
	\item For some special values of the degree one has $X^{[2]}\simeq S^{[2]}$, so that the second exceptional divisor is the Hilbert--Chow divisor of $S$. The first case where this happens non-trivially (i.e.\ $S\neq X$) is a very general $K3$ surface of degree $12$ (see for example \cite[\S3.3]{morettirojas}). In this case there is an extra option for the rational surface: one may take $P$ as the product of two rational curves in $S$. It would be interesting to see if analogous considerations of \autoref{remark} are possible.
	\item The first case where $S\neq X$ and the exceptional divisor of the contraction is a conic bundle, and not a $\mathbb P^1$-bundle, is in degree $16$, see \cite{vangeemenkapustka} and for an explicit construction \cite[\S4.1]{morettirojas}.
\end{enumerate}

We now turn to \autoref{primitivethm}, which describes $P$ more explicitly for a countable series of degrees. This section is independent of the proof of \autoref{mainthm}. Its purpose is geometric: it connects the hyperelliptic pencil on a curve to a rational surface in $X^{[2]}$ through two bundles on $X$. We refer the reader to \cite[Chapter 10, \S 2]{huybrechtsbook} and \cite{mukai} for general theory of stable bundles on $K3$ surfaces.

\begin{setup}\label{bundlesetup}
Let $(X,L)$ be a very general primitively polarized $K3$ surface of genus $g$, so $L^2=2g-2$ and $\Pic(X)=\bZ L$. Fix an integral nodal curve $C\in|L|$ of geometric genus $\gamma$, with $n=g-\gamma$ nodes. Let $i:\wC\to X$ be the normalization followed by the inclusion, and let $A\in\Pic^2(\wC)$ be the hyperelliptic pencil. For $0\leq j\leq\gamma-1$,
\[
\operatorname{Sym}^jH^0(A)\xrightarrow{\sim}H^0(A^{\otimes j})
\]
has dimension $j+1$.
\end{setup}

Ciliberto and Knutsen produce an irreducible two-dimensional family $T$ of such curves in the two numerical series of \autoref{primitivethm} \cite[Thm.~0.1]{cilibertoknutsen}. 

Set
\[
q=\begin{cases}
r-1,&g=r(r+1),\\
r,&g=(r+1)^2.
\end{cases}
\]

For $0\leq j\leq\gamma-1$, let $\cG_j\subseteq i_*A^{\otimes j}$ be the image of evaluation and define $M_j$ by
\begin{equation*}\label{kernelseq}
0\longrightarrow M_j
\longrightarrow\operatorname{Sym}^jH^0(A)\otimes\cO_X
\longrightarrow\cG_j
\longrightarrow0.
\end{equation*}
The sheaf $M_j$ is locally free of rank $j+1$. We use
\[
E=M_q,
\qquad
F=M_{q+1}.
\]

\begin{lem}\label{bundlefacts}
The bundles $E$ and $F$ are $\mu_L$-stable. For $j=q,q+1$ one has
\[
c_1(M_j)=-L,
\qquad
c_2(M_j)=2j+n.
\]
The bundle $E$ is rigid. If $g=r(r+1)$ the bundle $F$ is rigid as well; if $g=(r+1)^2$ it moves in a two-dimensional moduli space, which is a $K3$ surface.
\end{lem}

\begin{proof}
The bundle $M_j$ has no sections, while its dual is generated at the general point and its first Chern class generates $\mathrm{Pic}(X)$. This easily implies stability of $M_j$.

Let
\[
\delta_j=\operatorname{length}(i_*A^{\otimes j}/\cG_j).
\]
The sequence gives
\[
c_1(M_j)=-L,
\qquad
c_2(M_j)=2j+n-\delta_j,
\]
so the Mukai vector is $v(M_j)=(j+1,-L,\gamma-j+\delta_j)$ and
\[
v(M_j)^2=2g-2-2(j+1)(\gamma-j+\delta_j).
\]
Stability implies simplicity, hence $v(M_j)^2\geq-2$, that is $(j+1)(\gamma-j+\delta_j)\leq g$. In both series $(j+1)(\gamma-j)=g$ for $j=q$; for $j=q+1$ it equals $g$ if $g=r(r+1)$ and $g-1$ if $g=(r+1)^2$, in which case $j+1\geq3$. Thus $\delta_j=0$ in all cases, which gives the Chern classes, and $v(M_j)^2$ equals $-2$ or $0$ accordingly. By \cite{mukai} (see also \cite[Chapter 10, Prop. ~2.5]{huybrechtsbook}) a stable bundle with $v^2=-2$ is rigid and unique with its Mukai vector, and the moduli space of stable sheaves with a primitive isotropic Mukai vector is a $K3$ surface.
\end{proof}

Multiplication by a section $s\in H^0(A)$ gives a diagram with exact rows:
\[
\begin{tikzcd}[column sep=large]
0 \arrow{r} & E \arrow[hook]{r}\arrow{d}{\sigma_s} &
\operatorname{Sym}^{q}H^0(A)\otimes\cO_X
\arrow[two heads]{r}\arrow[hook]{d}{\cdot s} &
\cG_q\arrow{r}\arrow[hook]{d}{\cdot s} &0\\
0 \arrow{r} & F \arrow[hook]{r} &
\operatorname{Sym}^{q+1}H^0(A)\otimes\cO_X
\arrow[two heads]{r} &
\cG_{q+1}\arrow{r} &0.
\end{tikzcd}
\]
The right vertical arrow is well defined because multiplication by $s$ sends $\operatorname{Sym}^qH^0(A)$ into $\operatorname{Sym}^{q+1}H^0(A)$, and it is injective because $\cG_q\subseteq i_*A^{\otimes q}$ and $\wC$ is integral. The middle vertical map is injective with cokernel $\cO_X$. The snake lemma gives
\begin{equation}\label{IZsequence}
0\longrightarrow E\xrightarrow{\ \sigma_s\ }F
\longrightarrow\cI_{Z_s}\longrightarrow0,
\qquad
\cO_{Z_s}=\cG_{q+1}/s\,\cG_q,
\end{equation}
and by \autoref{bundlefacts} $\operatorname{length}(Z_s)=c_2(F)-c_2(E)=2$. Away from the nodes of $C$ the sheaf $\cG_j$ agrees with $i_*A^{\otimes j}$, because $A^{\otimes j}$ is base point free. Hence $Z_s=i(\operatorname{Div}(s))$ whenever $\operatorname{Div}(s)$ avoids the preimages of the nodes, and the subschemes $Z_s$ sweep out $C$ as $s$ moves in $|A|$. Thus the hyperelliptic pencil is a line
\[
|A|\subset\bP\Hom(E,F).
\]

For a general morphism $\sigma:E\to F$, injectivity and torsion-freeness of the cokernel are open conditions. Its cokernel has rank one and trivial determinant, so it is an ideal sheaf $\cI_{Z(\sigma)}$. Its Chern class shows directly that $Z(\sigma)$ has length two. We obtain a rational map
\begin{equation}\label{zetamap}
\zeta:\bP\Hom(E,F)\dashrightarrow X^{[2]},
\qquad
[\sigma]\longmapsto Z(\sigma),
\end{equation}
which is regular along $|A|$.

Let $\ell\subset\bP\Hom(E,F)$ be a line along which $\zeta$ is regular. The subschemes $Z(\sigma)$, $\sigma\in\ell$, form a double cover of $\ell\simeq\bP^1$ mapping to $X$. For $\ell=|A|$ this cover is integral and maps birationally onto $C$. Both properties are open, so they hold for a general member of any irreducible family of lines containing $|A|$. The image $D_\ell$ of such a line deforms $C$, hence lies in $|L|$; its normalization is a double cover of $\bP^1$ of genus at least $\gamma$ by semicontinuity, hence equal to $\gamma$ by maximality.

\begin{prop}\label{movingbundleprop}
Assume that two distinct general members of $T$ have isomorphic bundles $E$ and $F$. Then $X$ carries a rational covering family of curves in $|L|$ with hyperelliptic normalizations of geometric genus $\gamma$.
\end{prop}

\begin{proof}
The two curves give two lines in $\bP\Hom(E,F)$, which are distinct because the subschemes parametrized by each line sweep out the corresponding curve. Let $\Lambda$ be their span. The Grassmannian of lines in $\Lambda$ is rational, so it contains a rational curve $B$ through the two lines. By the discussion above, a general $\ell\in B$ gives a curve $D_\ell\in|L|$ with hyperelliptic normalization of geometric genus $\gamma$. The family $\{D_\ell\}_{\ell\in B}$ contains two distinct curves, so it is nonconstant and covers $X$.
\end{proof}

\begin{proof}[Proof of \autoref{primitivethm}]
By \autoref{bundlefacts}  the bundle $E$ is rigid and, therefore, the same for all members of $T$. If $g=r(r+1)$ the same holds for $F$, and \autoref{movingbundleprop} applies; this is the plane construction of \cite[Prop.~7.2]{flaminiknutsenpacienza}.

If $g=(r+1)^2$, the bundle $F_t$ attached to $t\in T$ moves in a $K3$ surface $\cM$. Consider the classifying map
\[
T\dashrightarrow\cM,
\qquad t\longmapsto F_t.
\]
If it is not generically injective, two distinct general members of $T$ share $F$, and $E$ as well, so \autoref{movingbundleprop} applies. If it is generically injective, it is birational, since $T$ and $\cM$ are irreducible surfaces. Rational curves are dense in $\cM$ \cite{cgl}; choose one meeting the open set on which the inverse map is defined. It gives a nonconstant rational family of the original hyperelliptic curves, which therefore covers $X$.
\end{proof}

The rational family of hyperelliptic curves is a double cover of a rational surface (taking the quotient by the fiberwise hyperelliptic involution) and induces a rational surface in $X^{[2]}$.
\section{Applications and questions}\label{sec:consequences}

\begin{proof}[Proof of \autoref{irrCor}]
	Let $P\subset X^{[2]}$ be a rational surface supplied by \autoref{mainthm}. Resolving $P$ and the restricted universal family, we obtain a double cover $\widetilde P'\to\widetilde P$ of a smooth rational surface and a dominant, generically finite evaluation map $\widetilde P'\dashrightarrow X$; in particular, $X$ is dominated by a double cover of $\bP^2$. This double cover is connected: otherwise a rational component would dominate $X$, and the $K3$ surface $X$ would be uniruled.
	
	Let $x_1,\ldots,x_p\in X$ be general points, with $p\geq2$, choose lifts $y_i\in\widetilde P'$, and let $u_i\in\widetilde P$ be their images. Through general points of a rational surface there pass rational curves of sufficiently high degree in a family with no base points outside the $u_i$; a general such curve $R$ through $u_1,\ldots,u_p$ meets the branch divisor transversely and avoids the exceptional and non-finite loci. Its inverse image is then an integral double cover of $R\simeq\bP^1$, which, $R$ being general, maps birationally onto its image in $X$. This image passes through $x_1,\ldots,x_p$, and it is neither rational nor elliptic since $p\geq2$; hence it is a hyperelliptic curve.
\end{proof}
{\bf Question.} Are these curves nodal? What is the minimal geometric genus of a hyperelliptic curve through $p$ general points?

\begin{proof}[Proof of \autoref{arithmeticIntro}]
	In all cases of the proof of \autoref{A}, the rational surface $P$ (or its strict transform, which is not contained in the flopping locus) is a rational Lagrangian surface in $M=X^{[2]}$, which is defined over $k$. Since rational Lagrangian surfaces of a given cohomology class in $M$ form a finite set (positive-dimensional families of Lagrangian subvarieties have irregular general member), $P$ is defined over $\bar k$, hence over a finite extension of $k$. Being geometrically rational, $P$ becomes rational over a further finite extension, so its rational points are Zariski dense. Each of them gives a length-two subscheme of $X$ defined over that field, supported on points of degree at most two, and these supports are dense in $X$ because the universal family over $P$ dominates $X$.
\end{proof}

{\bf Problem.} Give an interesting upper bound on the degree of the extension $k\subset k'$.
\subsection{Theorem D}
We first record the elementary part of the Beauville--Voisin argument that will be used in the incidence calculation.

\begin{lem}\label{divisordiagonal}
	Assume $\Pic(X)=\bZ L$. Then
	\begin{equation}\label{divisordiagonalformula}
		\Delta_{X*}L=L\times c_X+c_X\times L
		\qquad\text{in }\CH^3(X\times X).
	\end{equation}
\end{lem}

\begin{proof}
	The Bogomolov--Mumford theorem gives a rational curve $R\in |L|$; see \cite[Chapter 13 Thm.~1.1]{huybrechtsbook}.
	Push forward the identity
	\[
	\Delta_{\bP^1}=\bP^1\times p+p\times\bP^1
	\]
	along the normalization of $R$. This gives
	\[
	\Delta_{X*}R=R\times c_X+c_X\times R,
	\]
	which is \eqref{divisordiagonalformula}.
\end{proof}

The proof goes along the lines of \cite{bazhov}, to which the reader may compare.

\begin{proof}[Proof of \autoref{diagonalIntro}]
We assume $X$ of Picard rank $1$, the argument for the general $K3$ follows by a specialization argument.
	Let $P\subset X^{[2]}$ be a rational surface given by \autoref{mainthm}, let $S\to P$ be a resolution and let
	\[
	\mathcal Z\in\CH^2(S\times X)
	\]
	be the pullback of the universal family. It is finite flat of degree two over $S$. Denote by
	\[
	e:\mathcal Z\longrightarrow X
	\]
	the evaluation map and by $N>0$ its degree.
	
	Since $\CH_0(S)=\bZ$, the fibre $\mathcal Z_s$ has a fixed class
	\[
	\eta\in\CH_0(X),\qquad \deg(\eta)=2.
	\]
	We identify this class before computing the incidences. Give the image of the non-reduced locus its natural cycle structure
	\[
	C_P:=q_*\bigl(p^*[P]\cdot E_\Delta\bigr)\in\CH_1(X).
	\]
	Since $p^*E_{\mathrm{HC}}=2E_\Delta$, this is the image of $P\cap E_{\mathrm{HC}}$ with its natural multiplicities. It is a non-zero effective curve: indeed $C_P\cdot L=[P]\cdot h\delta$, which is positive for the surfaces of \autoref{mainthm} by a direct computation with their classes. If a point of this intersection represents a non-reduced subscheme supported at $x$, then
	$\eta=2[x].$
	Therefore $2([x]-[y])=0$ for points in the same component of $C_P$. Roitman's theorem \cite{roitman} implies that each component is a constant-cycle curve. An ample divisor supported on rational curves meets every component of $C_P$, and hence $\eta=2c_X$. In particular,
	\begin{equation}\label{diagonaleta}
		\Delta_{X*}\eta=2\,c_X\times c_X.
	\end{equation}
	From now on we write $c_i:=\operatorname{pr}_i^*c_X$ and $L_i:=\operatorname{pr}_i^*L$ on $X^3$.
	
	We now compute the two incidence cycles.
	Since $S$ is a smooth projective rational surface and $\Pic(X)=\bZ L$, the following (Künneth) decomposition holds:
	\[
	\CH^2(S\times X)=\operatorname{pr}_X^*\CH^2(X)\ \oplus\ \bigl(\Pic(S)\times L\bigr)\ \oplus\ [s_0]\cdot\operatorname{pr}_X^*\CH^0(X),
	\]
	where $s_0\in S$ is any point. Hence there is a unique divisor $D$ on $S$ such that
	\begin{equation}\label{universalTatedecomposition}
		[\mathcal Z]
		=[S]\times\eta
		+D\times L
		+N[s_0]\times[X]
		\qquad\text{in }\CH^2(S\times X).
	\end{equation}
	The last coefficient is $N$ because pushing forward to $X$ gives $N[X]$. The first coefficient is $\eta$: intersecting with $\{s\}\times X$ and pushing forward to $X$ kills the other two summands and gives the class of $\mathcal Z_s$. Put $\mu:=D\cdot D\in\bZ$.
	
	On $S\times X^2$, let $\mathcal Z_i$ be the pullback of $\mathcal Z$ from the $i$-th factor and define
	\[
	I_2:=(\operatorname{pr}_{X^2})_*(\mathcal Z_1\cdot\mathcal Z_2).
	\]
	Expanding \eqref{universalTatedecomposition} and retaining the terms of codimension two on $S$ gives
	\begin{equation*}\label{I2Tate}
		I_2=N(X\times\eta+\eta\times X)+\mu\,L\times L.
	\end{equation*}
	Over the open set where $\mathcal Z\to S$ is an \'etale double cover, its relative square is the union of the diagonal and the graph of the deck involution. If $\Sigma\in\CH^2(X\times X)$ denotes the push-forward of the second component, then
	\begin{equation*}\label{I2geometric}
		I_2=N\Delta_X+\Sigma.
	\end{equation*}
	Thus
	\begin{equation}\label{Sigmadecomposition}
		\Sigma=N(X\times\eta+\eta\times X-\Delta_X)+\mu\,L\times L.
	\end{equation}
	
	The third incidence contains the small diagonal. On $S\times X^3$ set
	\[
	I_3:=(\operatorname{pr}_{X^3})_*(\mathcal Z_1\cdot\mathcal Z_2\cdot\mathcal Z_3).
	\]
	The same expansion gives\footnote{Here and below, $\sum_{\mathrm{cyc}}$ denotes the sum over the three cyclic permutations of the factors of $X^3$; for instance, $\sum_{\mathrm{cyc}}c_2c_3=c_2c_3+c_3c_1+c_1c_2$.}
	\begin{equation}\label{I3Tate}
		I_3=4N\sum_{\mathrm{cyc}}c_2c_3
		+2\mu\sum_{\mathrm{cyc}}c_1L_2L_3.
	\end{equation}
	The relative cube of a double cover has four components: one where all three points agree and three where exactly two agree. Consequently
	\begin{equation}\label{I3geometric}
		I_3=N\Delta_{123}
		+\delta_{12*}\Sigma
		+\delta_{13*}\Sigma
		+\delta_{23*}\Sigma,
	\end{equation}
	where $\delta_{12}(x,y)=(x,x,y)$, $\delta_{13}(x,y)=(x,y,x)$ and $\delta_{23}(x,y)=(x,y,y)$. These cycle identities may be checked over the \'etale locus; the ramification locus produces no additional two-dimensional component.
	
	We compute $\delta_{ij*}\Sigma$ term by term from \eqref{Sigmadecomposition}: $\delta_{12*}(X\times\eta)=2\Delta_{12}\cdot c_3$, $\delta_{12*}\Delta_X=\Delta_{123}$, $\delta_{12*}(\eta\times X)=2c_1c_2$ by \eqref{diagonaleta}, and $\delta_{12*}(L\times L)=c_1L_2L_3+L_1c_2L_3$ by \autoref{divisordiagonal}. Summing over the three pairs gives
	\[
	\sum_{i<j}\delta_{ij*}\Sigma
	=2N\sum_{\mathrm{cyc}}\Delta_{12}\cdot c_3+2N\sum_{\mathrm{cyc}}c_2c_3-3N\Delta_{123}+2\mu\sum_{\mathrm{cyc}}c_1L_2L_3.
	\]
	Substituting into \eqref{I3geometric} and comparing with \eqref{I3Tate}, the terms involving $\mu$ cancel, and we obtain
	\begin{equation}\label{beforeThetacancellation}
		2N\Bigl([\Delta_{123}]-\sum_{\mathrm{cyc}}[\Delta_{12}]\cdot c_3+\sum_{\mathrm{cyc}}c_2c_3\Bigr)=0
		\qquad\text{in }\CH^4(X^3).
	\end{equation}
	Dividing by $2N$ proves the stated decomposition of the small diagonal.
\end{proof}
\begin{rem} \label{remark}
	The argument proving \autoref{diagonalIntro} only divides by $2N$, where $N$ is the degree of the incidence correspondence over $P$; all other steps hold with integral coefficients (up to replacing $\eta$ with $2c_X$). Hence the Beauville--Voisin relation holds in $\CH^4(X^3)$ up to $2N$-torsion, and up to $2\gcd_P N_P$-torsion when several surfaces $P$ are available. The reader may consult \cite{oberdieck,hassetttschinkelcones,bakkerplanes}  for the cohomology class of a Lagrangian plane from which the following sketched computation can be deduced. For a Lagrangian plane whose lines meet $E_{\mathrm{HC}}$ in $2t$ points one finds $N=(t^2-1)/8$. If $L^2=2$, the plane given by the fibres of the double cover has $N=1$. If $s\equiv 9,15\pmod{24}$ and $L^2=\tfrac52(5s^2-1)$, then $X^{[2]}$ contains a Lagrangian plane with $t=5s$ and incidence degree $\frac{25s^2-1}{8}$, and the surfaces of \autoref{mainthm} coming from the divisorial contraction have incidence degree coprime to it (see also \autoref{coprimm}). In all these cases the Beauville--Voisin decomposition of the small diagonal holds in $\CH^4(X^3)$ up to $2$-torsion.  The search of the degrees for which there are two correspondences of coprime degrees was done by AI.
\end{rem}

{\bf Question.} Does the decomposition of \autoref{diagonalIntro} hold with coefficients in $\mathbb Z [\frac{1}{2}]$ for any $K3$ surface?

\end{document}